\documentclass{amsart}

\usepackage{amssymb}
\usepackage{mathrsfs}
\usepackage{hyperref}
\usepackage[utf8]{inputenc}
\usepackage{verbatim}
\makeatletter
\newcommand{\oset}[3][0ex]{
  \mathrel{\mathop{#3}\limits^{
      \vbox to#1{\kern-2\ex@
      \hbox{$\scriptstyle#2$}\vss}}}}
\makeatother

\newtheorem{theorem}{Theorem}[section]
\newtheorem{thm}[theorem]{Theorem}
\newtheorem{lem}[theorem]{Lemma}

\newtheorem{cor}[theorem]{Corollary}

\newtheorem{conj}[theorem]{Conjecture}

\theoremstyle{definition}

\theoremstyle{remark}
\newtheorem{rem}[theorem]{Remark}

\numberwithin{equation}{section}

\newcommand{\FF}[1]{\mathbb F_{#1}}
\newcommand{\ZZ}{\mathbb Z}
\newcommand{\NN}{\mathbb N}

\newcommand{\sbs}{\subseteq}

\newcommand{\Hom}{\operatorname{Hom}}
\newcommand{\End}{\operatorname{End}}

\newcommand{\im}{\operatorname{Im}}
\newcommand{\diam}{\operatorname{diam}}
\newcommand{\Didiam}{\oset[-0.3ex]{\rightarrow}{\operatorname{diam}}}
\newcommand{\Tr}{\operatorname{Tr}}
\newcommand{\mc}{\mathcal C}

\newcommand{\mt}{\mathcal T}
\newcommand{\E}{E(\mt)}
\newcommand{\Gr}{\Gamma(\mt)}
\newcommand{\cdet}{\operatorname{det}_c}
\newcommand{\Pot}{\operatorname{Pot}}
\newcommand{\Cay}{\operatorname{Cay}}
\newcommand{\DiCay}{\oset[-0.3ex]{\rightarrow}{\operatorname{Cay}}}

\begin{document}
\title[Diameter of Cayley graphs of $SL(n,p)$]{Diameter of Cayley
  graphs of $SL(n,p)$ with generating sets containing a transvection}

\author{Zolt\'an  Halasi} 
\address{ Department of Algebra and Number Theory,
  E\"otv\"os University, P\'azm\'any P\'eter s\'et\'any 1/c, H-1117,
  Budapest, Hungary \and Alfr\'ed R\'enyi Institute of Mathematics,
  Re\'altanoda utca 13-15, H-1053, Budapest, Hungary\newline
  ORCID: \url{https://orcid.org/0000-0002-1305-5380}
} 
\email{zhalasi@caesar.elte.hu and halasi.zoltan@renyi.hu}

\date{\today}

\subjclass[2010]{}

\thanks{This work on the project
  leading to this application has received funding from the European
  Research Council (ERC) under the European Union's Horizon 2020
  research and innovation programme (grant agreement No. 741420). The
  author was also supported by the National Research, Development and
  Innovation Office (NKFIH) Grant No.~K115799 and by the J\'anos
  Bolyai Research Scholarship of the Hungarian Academy of Sciences.  }
\begin{abstract}
  A well-known conjecture of Babai states that if 
  $G$ is a finite simple group and $X$ is a generating set of $G$, then 
  the diameter of the Cayley graph $\Cay(G,X)$ is bounded 
  above by $(\log |G|)^c$ for some absolute constant $c$.
  The goal of this paper is to prove such a bound for
  the diameter of $\Cay(G,X)$ whenever $G=SL(n,p)$ and $X$ is a
  generating set of $G$ which contains a transvection.   
  A natural analogue of this result is also proved for 
  $G=SL(n,K)$, where $K$ can be any field. 
\end{abstract}
\maketitle
\section{Introduction}
For a finite group $G$ and a generator set $X\subset G$ the
(undirected) Cayley graph $\Cay(G,X)$ is the connected graph with
vertex set $G$ and with edge set $\{(g,gx)\,|\,g\in G, x\in X\}$.  The
diameter of this graph is the smallest $k$ such that every element of
$G$ can be written as a product of at most $k$ elements from $X\cup
X^{-1}$.  The diameter of a group $G$, denoted by $\diam(G)$, is the
maximum of the diameters of all Cayley graphs $\Cay(G,X)$ where $X$
runs through all the generating sets of $G$. 
The following conjecture was fomalised by Babai \cite[Conjecture 1.7]{Babai92}.
\begin{conj}\label{conj:Babai}
  If $G$ is a non-Abelian finite simple group, then 
  $\diam(G)<(\log|G|)^c$ for some absolute constant $c$.
\end{conj}
\begin{rem}
  One can similarly define the Cayley digraph $\DiCay(G,X)$ with set
  of directed edges $\{(g,gx)\,|\,g\in G, x\in X\}$ and the directed
  diameters $\diam(\DiCay(G,X))$ and $\Didiam(G)$. (Note that if $X$
  is symmetric, i.e. if $X=X^{-1}$, then $\DiCay(G,X)$ can be
  identified with $\Cay(G,X)$ in a natural way.)  Clearly,
  $\diam(\Cay(G,X)) \leq \diam(\DiCay(G,X))$.  On the other hand, a
  result of Babai \cite[Theorem 1.4]{Babai06} states that
  $\diam(\DiCay(G,X)) \leq O(\diam(\Cay(G,X))^2(\log |G|)^3)$ also
  holds. As a consequence, a positive answer to Conjecture
  \ref{conj:Babai} implies the same result for $\Didiam(G)$ (with a
  possibly larger $c$). Therefore, in the following we only consider
  undirected Cayley graphs $\Cay(G,X)$ with symmetric generating set
  $X$.
\end{rem}
Babai's conjecture was proved by Helfgott \cite{Helfgott}
for the case $G=SL(2,p)$.  Later, this
conjecture was verified for finite simple groups of Lie type of
bounded rank independently by Pyber and Szab\'o \cite{Pyber-Szabo} and
Breuillard, Green and Tao \cite{Breuillard_etal} In view of the
classification theorem, it remains to prove Babai's conjecture for
alternating groups and for classical groups of unbounded rank. Despite
the serious efforts, Babai's conjecture is still unsolved for both of
these classes.  So far, the best known general upper bounds for
$\diam(G)$ are the following.

On the one hand, let $G$ be an alternating group of degree $n$.
Then a quasipolynomial bound
$\diam(G)\leq \exp(O(\log n)^4\log\log(n))$
has been proved by Helfgott and Seress
\cite{Helfgott-Seress}. Later, their argument has been greatly simplfied in 
\cite{Helfgott-unified}. 
On the other hand, let $G$ be a classical group of rank $n$ over the $q$-element
field.  Then $\diam G\leq q^{O(n(\log n)^2)}$ by the main result of
\cite{H-Ma-Py-Youming}.

In case of $G=A_n$, an upper
bound $\diam(\Cay(G,X))=O(n^C)$ has been proved by Babai, Beals and Seress
\cite{Babai-Beals-Seress} under the restriction
that $X$ contains an element of degree $<n/(3+\varepsilon)$. (The
degree of a permutation is the number of elements moved by it.)  Under
this assumption, the authors managed to show that there is a $3$-cycle
which can be written as a product of $O(n^c)$ many elements from
$X$.  Then an upper bound $\diam(\Cay(G,X))=O(n^C)$ follows
trivially from this, since there are only $O(n^3)$ many $3$-cycles in
$A_n$. 

This, and similar results motivated Pyber to suggest a split of
Babai's conjecture into three subproblems in the classical case.  If
$G$ is a classical group with natural $KG$-module $V$, and $g\in G$,
then the \emph{support} of $g$ can be defined as the codimension of
the eigenspace of $g$ corresponding to the eigenvalue $1$ of $g$.
(Intuitively, small support of $g$ means that $g$ is close to being
the identity map.) Given a finite classical group $G$ of rank $n$ over the
$q$-element field and a generator set $X$ for $G$, a proof for Babai's
conjecture might be found by solving each of the following
subproblems:
\begin{itemize}
\item Find an element $1\neq g\in G$ whose length over $X$ is polynomial in 
$n(\log q)$ and whose support is at most $cn$ for some $c<1$.
\item Starting with the assumption of the existence $g\in X$ with support 
  $<cn$, find an element $1\neq t\in G$ whose length over $X$ is polynomial in 
$n(\log q)$ and whose support is as small as possible.
\item Starting with the assumption of the existence $1\neq t\in X$
  whose support is minimal in $G$, finish the proof of Babai's conjecture. 
\end{itemize}
The goal of this paper is to manage the third subproblem from this
list for the case $G=SL(V)$. To achieve this goal we need to consider
transvections.

Let $V$ be an $n>2$-dimensional vector space over an arbitrary field
$K$ and $G:=SL(V)$.  A \emph{transvection} $t\in SL(V)$ is an element
of the form $t=1+\nu$ where $\nu\in \End(V)$ has the property that
$\im(\nu)$ is a one-dimensional subspace in $\ker(\nu)$. Thus,
$\nu^2=0$ and $\dim(\ker(\nu))=n-1$. (Throughout this paper, $1\in
SL(V)$ denotes the identity map on $V$.)

Note that an element $1\neq t\in SL(V)$ has smallest support in 
$SL(V)\setminus \{1\}$ if and only if $t$ is a transvection. 
The number of transvections in $SL(V)$ is roughly $q^{2n-1}$,
so (unlike to the alternating case) even the third subproblem does not follow
trivially. 
\begin{thm}\label{thm:main_q}
  Let $V$ be an $n$-dimensional vector space over the finite field
  $\FF p$ where $p$ is a prime and let $X\subset G=SL(V)$ be a
  generating set of $SL(V)$ which contains a
  transvection. Then $\diam(\Cay(G,X))=O((\log p)^c n^{14})$ for some
  absolute constant $c$.
\end{thm}
\begin{rem}
  In this theorem, the constant $c$ is the same as in \cite[Main
  Theorem]{Helfgott}.  In fact, $O((\log p)^c)$ can be changed to a
  bound for the diameter of Cayley graphs of $SL(2,\FF p)$
  corresponding to generating sets $\{r,s\}$ where $r,s$
  is chosen to be two arbitrary non-commuting transvection in
  $SL(2,\FF p)$.
\end{rem}
For any transvection $t=1+\nu \in SL(V)$,
there is a unique \emph{transvection group} $t^K$ containing $t$,
which is defined as $t^K=\{t^\lambda\,|\,\lambda\in K\}$ where
$t^\lambda:=1+\lambda\nu$ for every $\lambda\in K$.  During the proof of Theorem
\ref{thm:main_q}, we also prove the following, which holds for any
field $K$.
\begin{thm}\label{thm:main_K}
  Let $V$ be an $n$-dimensional vector space over an arbitrary field $K$
  and let $X\subset G=SL(V)$ be a generating set of $SL(V)$ which
  contains a whole transvection group.
  Then $\diam(\Cay(G,X))=O(n^{12})$.
\end{thm}
In \cite{Humphries}, Humphries gave an exact condition vhen a set of
$n$ many transvections generate $SL(n,p)$. (Note that $n$ is the
minimal possible size of a set of such generators for $SL(n,p)$.)
Although we did not use Humphries result directly (his proof produces
an algorithm which only provides exponential bound to the diameter of 
the corresponding Cayley graph), Humphries' condition was very useful 
to find a proof for Theorem \ref{thm:main_K}. In fact, our proof for 
Theorem \ref{thm:main_K} also provides a generalisation and extension of 
Humphries' theorem. For details, see Section  \ref{sec:Humphries}.
\section{Notation and some basic tools}\label{sec:basic}
The purpose of this section is to introduce some terminology 
and to explain some very basic ideas used in the rest of the paper. 
Through this section let $K$ be any field and let $V$ be a
an $n$-dimensional vector space over $K$.

If $t=1+\nu\in SL(V)$ is any transvection, then 
it can be parametrised by $(d,\phi)\in V\times V^*$, where
$V^*=\Hom(V,K)$ is the dual space of $V$ such that 
a transvection $t=t_{d,\phi}$ satisfies $t(x)=x+\phi(x)d$ for every $x\in V$.
Note that this parametrisation is just almost unique, namely, 
$t_{\lambda\cdot d,\phi}=t_{d,\lambda\cdot \phi}$ holds for every $\lambda\in K$. 
Therefore,
the set of transvections can be identified with the elements
\[
  \{1+d\otimes \phi\,|\,d\otimes \phi\in V\otimes_K V^*,
  0\neq d\in V,\,0\neq \phi\in V^*\textrm{ and }\phi(d)=0\}.
\]
By fixing a basis $e_1,\ldots,e_n$ of $V$, we use the notation
$e_1^*,\ldots,e_n^*$ for the dual basis of $V^*$ satisfying
$e_i^*(e_j)=\delta_{ij}$. With help of these basises, elements of $V$
and $V^*$ can be identified with the set of column vectors and row
vectors over $K$ (each of length $n$), respectively.
We use the notation $[d]$ and $[\phi]^T$ for the corresponding
column and row vectors, respectively.

Under this identification, $d\otimes \phi$ is identified with the
usual dyadic product $[d]\cdot [\phi]^T$ and $V\otimes V^*$ is
identified with $K^{n\times n}$, the vector space of $n\times n$
matrices over $K$. Furthermore, $E_{ij}:=e_i\otimes e_j^*$ becomes the
usual basis of $K^{n\times n}$. 
Now, 
$0=\phi(d)=[\phi]^T[d]= \Tr([d]\cdot [\varphi]^T)$, which means
that the subspace of $V\otimes V^*$ generated by the set
$\{\nu=d\otimes \phi\,|\,1+\nu \textrm{ is a transvection}\}$
corresponds to the subspace $\{M\in K^{n\times n}\,|\,\Tr(M)=0\}$, so
it has dimenion $n^2-1$.

Consider the usual action of $SL(V)$ on $V$ and its dual action on
$V^*$, that is, $g\cdot\varphi(x):=\varphi(g^{-1}\cdot x)$ for every
$g\in SL(V),\ \varphi\in V^*$ and $x\in V$.  An easy calculation shows
that if $t=1+d\otimes \varphi$ is a transvection and $g\in SL(V)$,
then 
\begin{equation}\tag{Eq.\ 1}\label{eq:conj_tv_with_g}
  gtg^{-1}=1+(g\cdot d)\otimes(g\cdot\varphi).
\end{equation}  
This equation and the following three ones 
will be frequently used in this paper. 
\begin{lem}\label{lem:basic_op_eq}\leavevmode
  Let $r_1=1+v_1\otimes \phi_1$ and $r_2=1+v_2\otimes \phi_2$ be two
  transvections.  Then we have
  \begin{enumerate}
  \item[(a)]
    $r_2 r_1 r_2^{-1}=1+(v_1+\phi_2(v_1)\cdot v_2)\otimes
    (\phi_1-\phi_1(v_2)\cdot \phi_2)$. In particular, if
    $\phi_1(v_2)=0$ then
    $r_2 r_1 r_2^{-1}=1+(v_1+\phi_2(v_1)\cdot v_2)\otimes \phi_1$.
  \item[(b)] If $\phi_1(v_2)=0$, then $[r_2,r_1]=r_2r_1r_2^{-1}r_1^{-1}=
    1+\phi_2(v_1) v_2\otimes \phi_1$.
  \item[(c)] If $\phi_1=\phi_2$, then $r_1r_2=1+(v_1+v_2)\otimes\phi_1$.
  \end{enumerate}
\end{lem}
We next introduce the concept of \emph{transvection graphs}, 
which represents the relationship between any pairs of transvections. 
This tool will be crucial in our argument. 

For the remainder of this paper, let $\mt=\{1+d\otimes \phi\,|\,d\in
V,\,\phi\in V^*\}$ denote the set of all transvections. Occasionally, 
we allow ourselves to consider $1$ as an element of $\mathcal T$, 
and we think to $1$ as the trivial or the non-proper transvection.

Let $Y=\{t_1,\ldots,t_m\}\subset \mt$ be a set of transvections with
$t_i=1+\nu_i=1+d_i\otimes \phi_i$ for each $i$.  Then the directed
graph $\Gamma(Y)$ (called the transvection graph on $Y$) is defined as
follows.  Its vertex set is $Y\setminus\{1\}$ and there is a directed edge from
$t_i$ into $t_j$ if $\nu_j\nu_i\neq 0$, i.e. if $\phi_j(d_i)\neq 0$.

In particular, if $Y$ is the set of all transvections, we get the full
transvection graph $\Gr$.
Clearly, for every set $Y$ of transvections, $\Gamma(Y)$ is just the subgraph
of $\Gr$ induced by $Y\setminus\{1\}$.

For two transvections $r,s$, we say that $(r,s)$ is an edge if $(r,s)$
is a directed edge in $\Gr$.  The set of all edges (in $\Gr$) is
denoted by $\E$.  If $(r,s)$ is an edge, then we say that $(r,s)$ is
\emph{one-way directed} (resp. \emph{two-way directed}) if $(s,r)$ is
not an edge (resp. if $(s,r)$ is also an edge). Now, Lemma
\ref{lem:basic_op_eq} easily implies
\begin{lem}\label{lem:basic_op_graph}
  Let $x,y,z\in \mt$ be three proper transvections. 
  \begin{enumerate}
    \item[(a)] If $(x,y),(x,z)\notin \E$, then $(x,zyz^{-1})\notin \E$;
    \item[(b)] If $(x,y)\in \E$, but $(x,z)\notin \E$, then
      $(x,zyz^{-1})\in \E$.
    \item[(c)] If $(x,y),\,(y,z)\in \E$ but $(x,z)\notin\E$, then 
      $(x,yzy^{-1})\in \E$. 
    \item[(d)] If $(x,y),(x,z),(z,y)\in\E$, then
      $(x,z^\lambda yz^{-\lambda})\notin\E$ for a suitable $\lambda\in K$.
    \item[(e)] If $(x,y)$ is a one-way edge, then 
      $([x,y],z)\in \E\iff (x,z)\in \E$ and $(z,[x,y])\in \E\iff (z,y)\in \E$.
  \end{enumerate}
  In (a)-(d), dual statements can be obtained by reversing the
  direction of the edges.
\end{lem}
In what follows, we will freely use Lemmas \ref{lem:basic_op_eq} and
\ref{lem:basic_op_graph} without referring to them.

For a set $X\subset SL(V)$, we define
$X^k:=\{x_1x_2\cdots x_k\,|\,x_1,\ldots,x_k\in X\}$ for every
$k\in \NN$.  In the following we assume that $1\in X$. This can
clearly be assumed without loss of generality and it implies that
$X^k\subset X^l$ whenever $k<l$, which results some simplification in
the notation.

For any two sets $X,Y\subset SL(V)$ we denote by $\ell_X(Y)$ the
length of $Y$ over $X$, i.e.  $\ell_X(Y):=\min\{k\in\NN\,|\,Y\subset
X^{k}\}$ is the smallest number $k$ such that every element of $Y$
can be written as a product of at most $k$ many elements from $X$. (If
there is no such $k$, then $\ell_X(Y):=\infty$ or undefined.)

Note that using this notation, the conclusion of Theorem \ref{thm:main_K} 
can be rewritten as $\ell_X(SL(V))=O(n^{12})$. Furthermore, 
$\ell$ has the property that 
\begin{equation}
\ell_X(Z)\leq \ell_X(Y)\cdot\ell_Y(Z)\tag{Eq.\ 2}\label{eq:length_submult}
\end{equation}
for any three sets $X,Y,Z\subset SL(V)$. 
This property can be used to split the proof of Theorem
\ref{thm:main_K} into several steps by providing a chain of sets of
transvections with stronger and stronger properties such that each one
has a sufficiently small length over the previous one. The main goal
of the proof of Theorem \ref{thm:main_K} is to construct all elements
of $\mt$ as a short product of elements from $X$. In order to achieve
our goal, we need to ensure that the above mentioned sets of transvections 
have a property, which we call $K$-closed.

We say that a subset $Y\sbs \mt$ is $K$-closed if $t^K\sbs Y$ holds
for every $t\in Y$. If $Y$ is any set of transvections, then its
$K$-closure is defined as $Y^K=\{s\in \mt\,|\,s\in t^K\textrm{ for
  some }t\in Y\}$. Clearly $Y\sbs \mt$ is $K$-closed $\iff Y=Y^K$.

During our argument, we always generate new transvections in a way as
in Equation \ref{eq:conj_tv_with_g} and in Lemma
\ref{lem:basic_op_eq}. More precisely, we start with $t^K\subset
X$. In a general step, we have an already constructed $K$-closed
$Y\subset\mt$ whose length is known to be short enough in $X$ and
we construct a new element $z\in\mt$ by one of the following ways:
\begin{itemize}
  \item $z=gyg^{-1}$ where $y\in Y$ and $g\in X^k\cup Y$ 
    (for some short enough $k$); Then $z^K=gy^Kg^{-1}$.
  \item $z=[x,y]$ where $x,y\in Y$; Then $z^K=[x^K,y]$.
  \item $z=xy$ where $x,y\in Y$ and $\ker(x)=\ker(y)$;
    Then $z^\lambda=x^\lambda y^\lambda$ for every $\lambda\in K$.
\end{itemize}
This implies that the upper
bound we give for $\ell_{X\cup Y}(z)$ is also an upper bound for
$\ell_{X\cup Y}(z^K)$. As a consequence, along with $z$ we can add the
whole $z^K$ to $Y$. In this way, we can guarantee that the set of
already constructed transvections remains $K$-closed through the whole proof.
\section{The reduction of Theorem \ref{thm:main_q} to Theorem \ref{thm:main_K}}
In this section let $V$ be an $n$-dimensional vector space over $\FF p$. 
If $X$ is a symmetric generating set of $SL(V)$ and $t\in X$ is a
transvection, then we can form many transvections by repeatedly
conjugating with the elements of $X$. In that way we get sets of
transvections
\[
\mc_k=\mc(t,X,k):=
\{x_k\cdots x_1 t x_1^{-1}\cdots x_k^{-1}\,|\,x_1,\ldots x_k\in X\}
\subset X^{2k+1}\]
and corresponding directed graphs $\Gamma(\mc_k)$ for every $k$. 

\begin{thm}\label{thm:basis_in_C_n}
  Let $X$ be a symmetric generating set of $SL(V)$, which contains a
  transvection $t$ and let $\mc_n=\mc(t,X,n)=\{r_1=t,\ldots,r_m\}
  \subset X^{2n+1}$, where $r_i=1+v_i\otimes \phi_i$ with for every
  $1\leq i\leq m$.  Then $\langle v_1,\ldots,v_m\rangle=V$ and
  $\Gamma(\mc_n)$ contains a directed cycle;
\end{thm}
\begin{proof}
  Since $SL(V)$ acts irreducibly on $V$ and $X$ generates $SL(V)$,
  $\langle X^i(v_1)\,|\,i\in \NN\rangle=V$. Therefore, if $\langle
  X^i(v_1)\rangle=\langle X^{i+1}(v_1)\rangle$ for some $i$, then
  $\langle X^i(v_1)\rangle=V$ must hold. Thus, $\langle
  X(v_1)\rangle<\langle X^2(v_1)\rangle<\ldots$ is a chain of
  subspaces, which is strictly increasing until it reaches $V$. Since
  $\dim V=n$, it follows that $X^n(v_1)$ generates $V$. Using the
  Equality \ref{eq:conj_tv_with_g}, we get that 
  $\langle v_1,\ldots,v_m\rangle=V$.

  It follows from $\langle v_1,\ldots,v_m\rangle=V$ and from the definition of  
  $\Gamma(\mc_n)$ that there is no source vertex in $\Gamma(\mc_n)$. 
  Recall that a directed graph is acyclic, if it does not have a directed
  cycle.  As any finite acyclic graph has a source vertex, we
  get that $\Gamma(\mc_n)$ has a directed cycle, which proves our claim.
\end{proof}
Now, we can reduce Theorem \ref{thm:main_q} to Theorem
\ref{thm:main_K}.
\begin{thm}\label{thm:finding_a_tv_group}
  Let $V$ be an $n$-dimensional vector space over $\FF p$ where $p$ is
  a prime and let $X\subset SL(V)$ be a symmetric generating set
  containing a transvection $t$. Then $X^l$ contains a full transvection group
  for $l=O((\log p)^c\cdot n^{2})$.
\end{thm}
\begin{proof}
  By Theorem \ref{thm:basis_in_C_n}, $\Gamma(\mc_n)$ contains a
  directed cycle. Choosing a directed cycle of minimal length we get
  $r_1,\ldots,r_k\in \mc_n$ such that $(r_i,r_j)$ is a directed edge
  in $\Gamma_{\mc_n}$ if and only if $j-i\equiv 1\pmod k$. (In the
  following, when we consider a directed cycle of length $k$ we think
  of the indices as elements of $\ZZ_k$.)  Let
  $r_i=1+v_i\otimes \phi_i$ for every $1\leq i\leq k$, so
  $\phi_{j}(v_i)\neq 0\iff j-i\equiv 1\pmod k$.  Let
  $\alpha_1,\alpha_2\ldots,\alpha_k\in \FF p$.  Using the equality
  $\phi_{i+1}\Big(\sum_{s=1}^k \alpha_s v_s\Big)= \alpha_i\cdot
  \phi_{i+1}(v_i)$ for every $1\leq i\leq k$, we get that
  $v_1,\ldots,v_k\in V$ are linearly dependent, so $k\leq n$.

  Let
  $s_i:=r_ir_{i+1}\ldots r_{k-1}r_{k}r_{k-1}^{-1}\ldots
  r_{i+1}^{-1}r_i^{-1}$ for every $2\leq i\leq k-1$.  Using Lemma
  \ref{lem:basic_op_graph}/(a),(b),(c) or their duals repeatedly for
  $i=k-1,\ldots,2$, we get that $(r_{i-1},s_i)$ and $(s_i,r_1)$ are
  one way edges for $i>2$ and $(r_1,s_2)$ is a two-way edge. Clearly,
  $\ell_{\mc_n}(s_2)\leq 2k-3\leq 2n$, so $\ell_X(r_1,s_2)\leq
  2n^2$. Now, $\langle r_1,s_2\rangle\simeq SL(2,\FF p)$.  By
  \cite[Main Theorem]{Helfgott}, the diameter of $SL(2,\FF p)$ is
  $O((\log p)^c)$, so $X^{(O((\log p)^c n^2)}$ contains a full
  transvection group.
\end{proof}
\begin{cor}
  Theorem \ref{thm:main_K} implies Theorem \ref{thm:main_q}. 
\end{cor}
\begin{proof}
  Let $V$ be an $n$-dimensional vector space over $\FF p$ and let 
  $X\subset SL(V)$ be a symmetric generating set containing a
  transvection.  By Theorem \ref{thm:finding_a_tv_group}, $X':=X^k$
  contains a transvection group for $k=O((\log p)^c\cdot n^2)$.  Thus,
  Theorem \ref{thm:main_K} says that $\diam(\Cay(G,X'))=O(n^{12})$. Therefore,
  \[\diam(\Cay(G,X))\leq k\cdot \diam(\Cay(G,X')) =O((\log p)^c\cdot n^{14}).\]
\end{proof}
\section{Proof of Theorem \ref{thm:main_K}}\label{sec:main_K}
For the remainder, we assume that $V$ is an $n$-dimensional vector
space over an arbitrary field $K$ and $X$ is a symmetric generating
set of $SL(V)$ which contains a whole transvection group $t^K$.

First, we claim that it is enough to prove that the length of $\mt$ over $X$ is
bounded by a polynomial in $n$. In fact, we prove this for
a relatively small subset of transvections instead of $\mt$. 
\begin{lem}\label{lem:standard_tv}
  Let $e_1,\ldots, e_n$ be a basis of $V$ and
  let $Y=\{1+K\cdot e_i\otimes e_j^*\,|\,1\leq i,j\leq n,\; i\neq j\}$.
  Then we have $\diam(\Cay(G,Y))=O(n^2)$.
\end{lem}
\begin{proof}
  Identify $G=SL(V)$ with $SL(n,K)$ and $Y$ with the set
  $\{1+K\cdot E_{ij}\,|\,i\neq j\}$ as in Section \ref{sec:basic}. As
  it is well-known, every row operation on a matrix can be given by
  mupltiplying the matrix from the left with a product of constantly
  many elements from $Y$.
  The claim follows from the Gaussian elimination process.
\end{proof}
\begin{lem}\label{lem:strongly_connected}
  Let $Y_1$ be a $K$-closed set of transvections defined as
  \[
    Y_1=\mc(t^K,X,n^2)=\{x_{n^2}\cdots x_1t^\lambda x_1^{-1}\cdots x_{n^2}^{-1}
    \,|\,\lambda\in K,\,x_1,\ldots,x_{n^2}\in X\}\subset X^{2n^2+1}.
  \]
  Then
  \begin{enumerate}
  \item There are sets of transvections
    $\{s_1,s_2,\ldots,s_n\}\subset Y_1$ and
    $\{t_1,\ldots,t_n\}\subset Y_1$ with $s_i=1+a_i\otimes \alpha_i$ and
    $t_i=1+b_i\otimes \beta_i$ for each $i$
    such that $a_1,\ldots,a_n$ is a basis of
    $V$ and $\beta_1,\ldots,\beta_n$ is a basis of $V^*$.
  \item For every transvection $x\in \mt$, there are
    $1\leq i,j\leq n$ such that $(s_i,x)$ and $(x,t_j)$ are edges
    in $\Gamma(Y_1\cup x)$.
  \item $Y_1$ generates $SL(V)$.
  \item $\Gamma(Y)$ is strongly connected for any set of transvections
    $Y$ containing $Y_1$. 
  \end{enumerate}
\end{lem}
\begin{proof}
  (1) and (2) are clearly follows from the proof of Theorem
  \ref{thm:basis_in_C_n}. (Here we just reformulated them for the
  convenience of the reader.) 

  In order to prove that $Y_1$ generates $SL(V)$, 
  we consider the case $K\neq \FF 2$ first.  Let
  $W=\langle v\otimes \phi\in V\otimes V^*\,|\, \phi(d)=0\rangle$, so
  $\dim(W)=n^2-1$.  For every $i$, let
  $W_i=\langle \nu\in V\otimes V^*\,|\,
  1+\nu\in\mc(t^K,X,i)\rangle\leq V\otimes V^*$.  Since $X$ generates
  $SL(V)$ and the transvection subgroups are all conjugate to each
  other, we get that $\cup_{i=1}^\infty W_i=W$.  As in the proof of
  Theorem \ref{thm:basis_in_C_n}, we get that $W_1<W_2<\ldots $ is a
  strictly increasing chain of subspaces until it reaches $W$. This
  means that $W_{n^2-1}=W$. Identifying $V\otimes V^*$ with $\End(V)$,
  one can see that there is no proper $W_{n^2-1}=W$-invariant subspace
  of $V$, which implies the same for the subgroup
  $H=\langle \mc(t^K,X,n^2-1)\rangle\leq SL(V)$. In other words, 
  $H$ is an irreducible subgroup of $SL(V)$, which is generated by 
  transvection groups. As a special case of the main result of 
  \cite{McLaughlin}, it follows that $H\geq Sp(V)$. 
  Since $X$ generates $SL(V)$ and $Sp(V)$ is not normal in $SL(V)$ (unless 
  $Sp(V)=SL(V)$),  
  we have $\langle x Sp(V)x^{-1}\,|\,x\in X\rangle=SL(V)$. 
  Thus, $\langle Y_1\rangle \geq \langle xHx^{-1}\,|\,x\in X\rangle 
  \geq SL(V)$ as claimed.
  
  The case when $K=\FF 2$ is similar, but we take the strictly
  increasing chain of subgroups $H_0=t^K< H_1< H_2\leq \ldots$ where
  $H_i=\langle \mc(t^K,X,i)\rangle$ for each $i$.  Now, the length of this
  chain can be bounded above by $\log_2(|SL(V)|)\leq n^2$.

  Now, let us assume that $\Gamma(Y)$ is not strongly connected for
  some set of transvections $Y\supset Y_1$.  Then there is a
  $\emptyset\neq Z\subsetneq Y$ such that there is no outgoing edge
  from $Z$ in $\Gamma(Y)$. This means that the subspace
  $V_Z=\langle v\,|\,\exists \phi\in V^*:\ 1+v\otimes \phi\in
  Z\rangle$ is a proper subspace of $V$ fixed by each element of
  $Y$. This contradicts with the irreducibility of
  $SL(V)=\langle Y\rangle$ on $V$.
\end{proof}
\begin{rem}
  For $k=2$, the statement analogous to the result of
  \cite{McLaughlin} does not hold. For this case, there are several
  other types of irreducible subgroups of $SL(V)$ which are generated
  by transvection groups exists. For a complete list, see
  \cite{McLaughlin_F2}.
\end{rem}
Our next goal is to produce a $K$-closed set of transvections $Y_2\supset
Y_1$ such that $Y_2$ contains a one-way directed edge and
$\ell_{Y_1}(Y_2)\leq O(n)$.  Of course, if $Y_1$ itself contains
a one-way directed edge, then we can choose $Y_2=Y_1$. Therefore, we assume
that every edge is two-way directed in $Y_1$. 

To achieve our goal, we consider cycles in $\Gamma(Y_1)$. 
Let $(r_1,r_2,\ldots,r_k)$ be a (two-way directed) cycle in $\Gamma(Y_1)$
with $k\geq 3$ and $r_i=1+v_i\otimes \phi_i$ for each $i$. 

We say that this cycle is \emph{non-singular} if 
\begin{equation}\tag{Eq.~3}\label{eq:cyclic-det}
  \cdet(v_1,\phi_1,\ldots,v_k,\phi_k):=
  \prod_{i=1}^k\phi_i(v_{i+1})+(-1)^{k-1}\cdot\prod_{i=1}^k\phi_{i+1}(v_i)\neq 0.
\end{equation}
\begin{rem}\label{rem:cyclic-det}\leavevmode
  \begin{enumerate}
  \item Note that the non-singularity of a cycle $(r_1,r_2,\ldots,r_k)$
    only depends on the transvection groups $r_1^K,\ldots,r_k^K$.
    Indeed, by changing $(1+v_i \otimes \phi_i)$ to
    $(1+v_i\otimes\phi)^\lambda=1+(\lambda v_i)\otimes
    \phi_i=1+v_i\otimes(\lambda\phi_i)$ for some
    $\lambda\in K^\times$, the value of
    $\cdet(v_1,\phi_1,\ldots,v_k,\phi_k)$ is multiplied by $\lambda$.
  \item The formula in (\ref{eq:cyclic-det}) can also be defined for any 
    set of transvections $r_1,\ldots,r_k$. Clearly, its value is zero 
    unless at least one of $(r_1,\ldots,r_k)$ and $(r_k,\ldots,r_1)$ is 
    a directed cycle in $\Gamma(\mt)$, while it 
    is non-zero if exactly one of $(r_1,\ldots,r_k)$ and $(r_k,\ldots,r_1)$ is 
    a directed cycle in $\Gamma(\mt)$.
  \item In the particular case when $k$ is odd and
    $(r_1,\ldots,r_k)$ is a chordless (or induced) cycle,
    $\cdet(v_1,\phi_1,\ldots,v_k,\phi_k)$ has a special meaning: It is
    just the determinant of the $k\times k$-matrix $(\phi_i(v_j))$.
  \end{enumerate}
\end{rem}
\begin{lem}\label{lem:non-sing-to-one-way}
  Let $C=(r_1,\ldots,r_k)$ be a chordless non-singular cycle in
  $\Gamma(Y_1)$ (with $k\geq 3$). Then $\Gamma(C^{2n+1})$ contains a
  one-way directed edge.
\end{lem}
\begin{proof}
  Let $r_i=1+v_i\otimes \phi_i$ for each $1\leq i\leq k$. By our assumption, 
  $\phi_{i+1}(v_i)\neq 0$ and $\phi_i(v_{i+1})\neq 0$ for each $1\leq i\leq k$,
  while $\phi_i(v_j)=0$ if $i-j\not\equiv \pm 1 \pmod k$. 
  
  First, let us assume that $k=3$.
  We calculate the conjugate 
  \[
  r_1'(\lambda):=
  r_2^\lambda r_1 r_2^{-\lambda}=1+(v_1+\lambda\phi_2(v_1) v_2)\otimes
  (\phi_1-\lambda\phi_1(v_2)\phi_2).
  \]
  Now, $(r_1',r_3)$ is a one-way directed edge in
  $\Gamma(C\cup\{r_1'\})$ if and only if both of the following
  inequality and equality hold:
  \begin{align*}
    \phi_3(v_1+\lambda\phi_2(v_1) v_2)=
    \phi_3(v_1)+\lambda\phi_2(v_1)\phi_3(v_2)\neq 0,\\
    (\phi_1-\lambda\phi_1(v_2)\phi_2)(v_3)=
    \phi_1(v_3)-\lambda\phi_1(v_2)\phi_2(v_3)=0.
  \end{align*}
  Since each $\phi_i(v_j)\neq 0$, there is such a $\lambda\in
  K^\times$ if and only if
  \[
  \begin{vmatrix}
    \phi_3(v_1)&\phantom{-}\phi_2(v_1)\phi_3(v_2)\\
    \phi_1(v_3)&-\phi_1(v_2)\phi_2(v_3)
  \end{vmatrix}=-\cdet(v_1,\phi_1,v_2,\phi_2,v_3,\phi_3)\neq 0,
  \]
  which exactly means that the cycle $(r_1,r_2,r_3)$ is
  non-singular. Thus, there is a one-way directed edge in
  $\Gamma(C^3)$.

  Now, we turn to the general case, and we use induction on $k$. Using
  the same argument as in the proof of Theorem
  \ref{thm:finding_a_tv_group}, we get that $v_1,\ldots,v_{k-2}\in V$ is
  linearly independent, so $k\leq n+2$ holds.  Let $r_{k-1}'$ be the
  conjugate of $r_k$ by $r_{k-1}$, so
  \[
  r_{k-1}'=r_{k-1}\cdot r_k\cdot r_{k-1}^{-1}
  =1+(v_k+\phi_{k-1}(v_k)v_{k-1})\otimes(\phi_k-\phi_k(v_{k-1})\phi_{k-1}),
  \]
  that is, $r_{k-1}'=1+(v_{k-1}')\otimes (\phi_{k-1}')$ with
  $v_{k-1}'=v_k+\phi_{k-1}(v_k)v_{k-1}$ and
  $\phi_{k-1}'=\phi_k-\phi_k(v_{k-1})\phi_{k-1}$ Thus, for every
  $1\leq i\leq k-2$, we have $(r_{k-1}',r_i)$ is an edge in
  $\Gamma(C\cup \{r_{k-1}'\})$ if and only if
  \[
  \phi_i(v_{k-1}')=\phi_i(v_k)+\phi_{k-1}(v_k)\phi_i(v_{k-1})\neq 0
  \iff i\in\{1,k-2\}.\]
  Similarly, $(r_i,r_{k-1}')$ is an edge if and only if $i\in\{1,k-2\}$. Thus, 
  $(r_1,\ldots,r_{k-2},r_{k-1}')$ is a chordless cycle of length $k-1$. 
  Furthermore, 
  \begin{align*}
    &\phi_1(v_{k-1}')=\phi_1(v_k),&&\phi_{k-2}(v_{k-1}')=
    \phantom{-}\phi_{k-1}(v_k)\phi_{k-2}(v_{k-1}),\\
    &\phi_{k-1}'(v_1)=\phi_k(v_1),&&\phi_{k-1}'(v_{k-2})=
    -\phi_k(v_{k-1})\phi_{k-1}(v_{k-2}),\\
  \end{align*}
  which implies that
  \[
  \cdet(v_1,\phi_1,\ldots,v_{k-2},\phi_{k-2},v_{k-1}',\phi_{k-1}')
  =\cdet(v_1,\phi_1,\ldots,v_{k-1},\phi_{k-1},v_{k},\phi_{k})\neq 0.
  \]
  Using this process repeatedly, we find shorter and shorter
  non-singular chordless cycles $(r_1,\ldots,r_{i-1},r_i')$ where 
  $r_i'=r_ir_{i+1}'r_i^{-1}$ for $i=k-1,k-2,\ldots,3$. In that way we can find a 
  non-singular two-way directed cycle 
  $(r_1,r_2,r_3')$, where 
  \[r_3'=r_3r_4\cdots r_{k-1}r_kr_{k-1}^{-1}\cdots r_{3}^{-1}\in
  C^{2n-1}.\] By the first part of the proof, $\Gamma(C^{2n+1})$
  contains a one-way directed edge.
\end{proof}
\begin{lem}\label{lem:Y1-non-singular}
  $\Gamma(Y_1)$ contains a non-singular chordless cycle. 
\end{lem}
\begin{proof}
  First, we reformulate the concept of non-singularity given by
  (\ref{eq:cyclic-det}) by introducing the \emph{potential} for
  two-way directed cycles. First, for any two transvections, 
  $1+c\otimes \gamma,\ 1+d\otimes \delta$ connected by a two-way directed edge
  let
  \[
  r(c,\gamma,d,\delta):=\frac{\delta(c)}{\gamma(d)}
  \]
  Now, let $(r_1,\ldots,r_k)$ be a two-way directed cycle where 
  $r_i=1+v_i\otimes \phi_i$ for each $i$. Then its potential is defined as 
  \[
  \Pot(r_1,\ldots,r_k):=\prod_{i=1}^k r(v_i,\phi_i,r_{i+1},\phi_{i+1}).
  \]
  Note that unlike to the above definition of $\cdet$ and $r$, the
  potential depends only on the transvection groups
  $r_1^K,\ldots,r_k^K$ and not on the particular choice of the $v_i$
  and the $\phi_i$. Clearly, the two-way directed cycle
  $(r_1,\ldots,r_k)$ is singular if and only if 
  $\Pot(r_1,\ldots,r_k)=(-1)^k$.

  The above function $r(c,\gamma,d,\delta)$ has the property
  \[
  r(d,\delta,c,\gamma)=\frac{1}{r(c,\gamma,d,\delta)}
  \]
  which can be used to calculate the potential of the symmetric
  difference of two two-way directed cycles glued by a subpath.  More
  concretely, let $(r_1,r_2,\ldots, r_k)$ be a two-way directed cycle
  and let us assume that for some $1\leq i<j\leq k$ there is a two-way
  directed path $r_i,q_1,\ldots,q_l,r_j$. Then we have
  \begin{align*}\tag{Eq. 4}\label{eq:Pot-glued}
  \Pot(r_1,\ldots,r_k)&=\Pot(r_1,\ldots,r_i,q_1,\ldots,q_l,r_j,\ldots,r_k)\\
  &\cdot\Pot(r_i,\ldots,r_j,q_l,\ldots,q_1)
  \end{align*}
  In particular, a cycle obtained by gluing two singular cycles
  is singular itself. 

  Let $Y\supset Y_1$ be any set of transvections and let us assume by
  a way of contradiction that every cycle in $Y$ is singular. (Note
  that this implies that $Y$ does not contain any one-way directed
  edge by Remark \ref{rem:cyclic-det}/(2) and by the strongly
  connected property of $Y$.)  Let $s_1=1+w_1\otimes
  \psi_1,\;s_2=1+w_2\otimes \psi_2\in Y$ be two neighbouring vertices
  in $\Gamma(Y)$ and let $s_3=s_2s_1s_2^{-1}=1+w_3\otimes \psi_3$
  where $w_3=w_1+\psi_2(w_1)w_2$ and
  $\psi_3=\psi_1-\psi_1(w_2)\psi_2$. Then $(s_1,s_3)$ and $(s_2,s_3)$
  are (two-way directed) edges. Let $t=1+u\otimes \mu\in
  Y$ be any transvection in $Y$. A small calculation shows that
  \begin{align*}
  \cdet(w_3,\psi_3,w_1,\psi_1,u,\mu)
& =\psi_3(w_1)\psi_1(u)\mu(w_3)+\psi_1(w_3)\mu(w_1)\psi_3(u)\\
&\hspace{-2cm}=-\psi_1(w_2)\psi_2(w_1)\psi_1(u)(\mu(w_1)+\psi_2(w_1)\mu(w_2))\\
&\hspace{-2cm}\phantom{=}+\psi_2(w_1)\psi_1(w_2)\mu(w_1)
              (\psi_1(u)-\psi_1(w_2)\psi_2(u))\\
&\hspace{-2cm}=-\psi_1(w_2)\psi_2(w_1)\Big(\psi_1(u)\psi_2(w_1)\mu(w_2)
              +\mu(w_1)\psi_1(w_2)\psi_2(u)\Big)\\
&\hspace{-2cm}=-\psi_1(w_2)\psi_2(w_1)\cdot\cdet(w_1,\psi_1,w_2,\psi_2,u,\mu)=0.
  \end{align*}
  Thus, if $(s_3,s_1,t)$ is a cycle, then it is singular. 
  A similar calculation shows that 
  if $(s_3,s_2,t)$ is a cycle, then it is singular, as well. 

  Now, let $r_1,\ldots,r_k\in Y$ such that $(r_1,\ldots,r_k,s_3)$ is a cycle in
  $\Gamma(Y\cup\{s_3\})$ not contained in $\Gamma(Y)$.  Since
  $(s_3,r_1)$ and $(r_k,s_3)$ are edges, both $r_1$ and $r_k$ are
  connected with at least one of $s_1$ and $s_2$. If, for example,
  both of them are connected with $s_1$, then $(s_1,r_1,s_3)$ and 
  $(s_1,s_3,r_k)$ both are singular cycles, so 
  $\Pot(s_3,r_k,s_1)=\Pot(s_3,s_1,r_1)=-1$.
  By using
  (\ref{eq:Pot-glued}) twice we get
  \[
    \Pot(r_1,\ldots,r_k,s_1)=\Pot(r_1,\ldots,r_k,s_3)\cdot 
    \Pot(s_3,r_k,s_1)\cdot\Pot(s_3,s_1,r_1).
  \]
  Furthermore, $(r_1,\ldots,r_k,s_1)$ is a cycle in $\Gamma(Y)$, so it
  is singular by our assumption. Therefore,
  $\Pot(r_1,\ldots,r_k,s_3)=(-1)^{k+1}$, which means that
  $(r_1,\ldots,r_k,s_3)$ is also singular.  Similarly, if, say, $r_1$ is
  connected with $s_1$ and $r_k$ is connected with $s_2$ then by using
  (\ref{eq:Pot-glued}) three times we get
  \begin{align*}
  \Pot(s_1,r_1,\ldots,r_k,s_2)&=\Pot(s_1,r_1,s_3)\cdot\Pot(s_1,s_3,s_2)\\
  &\cdot\Pot(s_2,s_3,r_k)\Pot(r_1,\ldots,r_k,s_3),
  \end{align*}
  which implies that $\Pot(r_1,\ldots,r_k,s_3)=(-1)^{k+1}$, that is, 
  $(r_1,\ldots,r_k,s_3)$ is singular again. 

  Now, let us assume that all cycles in $\Gamma(Y_1)$ are singular.
  The above argument shows that if we repeatedly add new transvections
  to $Y_1$ by conjugating previous ones with each other, then
  we can never get a non-singular cycle. But $Y_1$ generates $SL(V)$
  and all transvections are conjugate in $SL(V)$, which implies that
  sooner or later we get all the transvections of $SL(V)$ that
  way. Since $n\geq 3$, the graph $\Gamma(\mt)$ contains a
  non-singular cycle, which is a contradiction.

  Thus, we proved that $\Gamma(Y_1)$ contains a non-singular cycle.
  Let $(r_1,\ldots,r_k)$ be a non-singular cycle such that $k$ is as
  small as possible. We claim that $(r_1,\ldots,r_k)$ is chordless.
  Otherwise, there is a (two-way directed) edge $(r_i,r_j)$ in
  $\Gamma(Y_1)$ for some $1\leq i<j-1$. 
  Using (\ref{eq:Pot-glued}), we get that at least one of the shorter 
  cycles $(r_1,\ldots,r_i,r_j,\ldots,r_k)$ and $(r_i,r_{i+1},\ldots,r_j)$ 
  must be non-singular, a contradiction.
\end{proof}
Using Lemmas \ref{lem:Y1-non-singular} and \ref{lem:non-sing-to-one-way}
we get
\begin{cor}\label{cor:Y2-one-way-edge}
  There is a $K$-closed set of transvections $Y_2\supset Y_1$ with
  $\ell_{Y_1}(Y_2)\leq O(n)$ such that $\Gamma(Y_2)$ contains a
  one-way directed edge. 
\end{cor}
\begin{lem}\label{lem:Y3-short-diameter}
  There is a $K$-closed set of transvections $Y_3\supset Y_2$ of length 
  $\ell_{Y_2}(Y_3)\leq O(n)$ such that for every $s,t\in \mt$
  there is a directed path from $s$ into $t$ in $Y_3\cup \{s,t\}$ of
  length at most $2$.
\end{lem}
\begin{proof}
  Let $s,t\in\mt$ be two transvections. Since $Y_2\cup\{s,t\}$ is
  strongly connected, there is a directed path from $s$ into $t$ in
  $Y_2\cup\{s,t\}$. Let $s,r_1,\ldots,r_k,t$ be a directed path of
  shortest length with $r_i=1+v_i\otimes\phi_i$ for every $1\leq i\leq
  k$, so there is no edge from $r_i$ into $r_j$ if $1\leq i<j\leq k$.
  As in the proof of Theorem \ref{thm:finding_a_tv_group}, we get that
  $\{v_1,\ldots,v_k\}\subset V$ is linearly independent set, so $k\leq
  n$.  Let $r_{s,t}:=r_k\ldots r_2r_1r_2^{-1}\ldots r_k^{-1}$.  Using
  Lemma \ref{lem:basic_op_graph} repeatedly, we get that $s,r_{s,t},t$
  is a path.  Clearly, $\ell_{Y_2}(r_{s,t})\leq 2k-1\leq 2n$ also
  holds. Thus, the set
  \[
  Y_3:=Y_2\cup\{r_{s,t}^K\,|\,s,t\in \mt,\ (s,t)\textrm{ is not an edge}\}
  \]
  has the required properties. 
\end{proof}
\begin{lem}\label{lem:Y4-one-way-neighbours}
  There is a $K$-closed set of transvections $Y_4\supset Y_3$ with
  $\ell_{Y_3}(Y_4)\leq O(1)$ such that for every
  transvection $r\in \mt$, there are one-way edges $(r,e_r)$ and
  $(s_r,r)$ for some $e_r,s_r\in Y_4$.
\end{lem}
\begin{proof}
  Let $r\in \mt$ be any transvection and let $(s,t)$ be a one-way
  directed edge in $\Gamma(Y_2)$ (whose existence is guaranteed by
  Corollary \ref{cor:Y2-one-way-edge}). Let $r=r_0,\ldots, r_k=s$
  be a path of shortest length in $\Gamma(\{r\}\cup Y_3$. By the
  construction of $Y_3$, we have $0\leq k\leq 2$. If $k=0$, then
  $(r,t)$ is one-way directed. Now, let $k=1$. Using parts of Lemma
  \ref{lem:basic_op_graph}, we get that if $(t,r)$ is not an edge,
  then $(r,[s,t])$ is a one-way directed edge, while if $(t,r)$ is an
  edge, then $(r,t^{\lambda}st^{-\lambda})$ is a one-way directed edge
  for some $\lambda\in K$. Finally, if $k=2$, then we apply the $k=1$
  twice: First, we can apply it to $r_1,s,t$ we get a one-way directed
  edge $(r_1,e_{r_1})$, then we can apply it to $(r,r_1,e_{r_1})$ to
  get a one-way directed edge $(r,e_r)$.
 
  The existence of a suitable transvection $s_r$ can be proved in an
  analogous way.
\end{proof}
\begin{lem}\label{lem:bounded_graphs}
  Let $r_1=1+v_1\otimes\phi_1,\ldots,r_k=1+v_k\otimes \phi_k$ be a
  directed path of transvections where $k\leq 5$. Let us assume that
  at least one of $(r_1,r_2)$ and $(r_2,r_3)$ is one-way directed for
  $k=3$, while both of $(r_1,r_2)$ and $(r_{k-1},r_k)$ are one-way
  directed for $4\leq k\leq 5$. Furthermore, let
  $Z=\{r_1^K,\ldots,r_k^K\}$.  Then there is a $\psi\in V^*$ such that
  $s:=1+(v_1+v_k)\otimes \psi$ is a transvection with
  $\ell_Z(s^K)\leq c$ for some constant $c$.  Similarly, there is a transvection
  $t=1+w\otimes (\phi_1+\phi_k)$ such that $\ell_Z(t)\leq c$.
\end{lem}
\begin{proof}
  First, if $(r_1,r_k)$ or $(r_k,r_1)$ is an edge, then
  $s=r_k^{1/\phi_k(v_1)}r_1r_k^{-1/\phi_k(v_1)}$ or
  $s=r_1^{1/\phi_1(v_k)}r_kr_1^{-1/\phi_1(v_k)}$ satisfy the claim.
  So for the remainder, we assume that there is no edge between $r_1$
  and $r_k$ in either direction.
  
  If $k=3$ and, say, $(r_1,r_2)$ is a one-way directed edge, then let
  $s_1:=[r_2^{1/\phi_2(v_1)},r_1]=1+v_2\otimes\phi_1$. Now,
  $(s_1,r_3)$ is a one-way directed edge, so
  $s_2:=[r_3^{1/\phi_3(v_2)},s_1]=1+v_3\otimes \phi_1$. Therefore,
  $s:=r_1\cdot s_2=1+(v_1+v_3)\otimes \phi_1$. The case when
  $(r_2,r_3)$ is one-way directed can be handled in an analogous way.

  Now, let $k=4$ and let both $(r_1,r_2)$ and $(r_3,r_4)$ be one-way
  directed.  Choosing
  $s_1=[r_2^{1/\phi_2(v_1)},r_1]=1+v_2\otimes \phi_1$, the case $k=3$
  can be applied to the path $s_1,r_3,r_4$.

  Finally, let $k=5$ and let us assume that both $(r_1,r_2)$ and
  $(r_4,r_5)$ are one-way directed. If $(r_1,r_4)$ is an edge, then the
  case $k=3$ can be applied to the path $r_1,r_4,r_5$. If $(r_4,r_1)$
  is one-way directed, then choosing
  $s_1=[r_1^{1/\phi_1(v_4)},r_4]=1+v_1\otimes \phi_4$ and
  $s_2=[r_5^{1/\phi_5(v_4)},r_4]=1+v_5\otimes \phi_4$ we get
  $s:=s_1\cdot s_2=1+(v_1+v_5)\otimes \phi_4$. So for the remainder,
  we assume that there is no edge between $r_1$ and $r_4$ in either
  direction. Now, if $(r_3,r_1)$ is not an edge then choosing
  $s_1=[r_2^{1/\phi_2(v_1)},r_1]=1+v_2\otimes \phi_1$, the case $k=4$
  can be appplied to the path $s_1,r_3,r_4,r_5$. The case when
  $(r_5,r_3)$ is similar. Finally, let us assume that both $(r_5,r_3)$
  and $(r_3,r_1)$ are edges. Then there is a $\lambda\in K$ such that
  $s_1:=r_4^\lambda r_3r_4^{-\lambda}$ satisfies that $(r_5,s_1)$ is
  one-way directed (If $(r_5,r_3)$ is one-way directed, then
  $\lambda=0$ gives $s_1=r_3$). Now, the case $k=3$ can be applied to
  the path $r_5,s_1,r_1$.

  The existence of a $t=1+w\otimes (\phi_1+\phi_k)$ with
  $\ell_Z(t^K)\leq c$ can be proven in an analogous way.  
\end{proof}
\begin{rem}
  One can check that the above argument provides a value $c=28$.
\end{rem}
\begin{lem}\label{lem:Y5-everything-is-image}
  There is a $K$-closed set of transvections $Y_5\supset Y_4$ with
  $\ell_{Y_4}(Y_5)\leq n^6$ such that 
  for every $0\neq v\in V$ there is a $\phi\in V^*$ such that 
  $1+v\otimes \phi\in Y_5$. Similarly, for every $\psi\in V^*$ there is a
  $u\in V$ with $1+u\otimes \psi\in Y_5$.
\end{lem}
\begin{proof}
  We only prove the first claim, the second one can be proved in a
  similar way.  Let $0\neq v\in V$ be fixed.  By Theorem
  \ref{lem:strongly_connected}, there are transvections
  $s_1=1+a_1\otimes \alpha_1,\ldots,s_k=1+a_k\otimes \alpha_k\in Y_1$
  such that $k\leq n$ and $v=\sum_{i=1}^k a_i$. We prove the existence
  of a $\phi\in V^*$ such that $\ell_{Y_4}(1+v\otimes \phi)\leq k^6$
  by using induction on $k$. The claim is clearly true for $k=1$.  For
  an arbitrary $k\leq n$ let $v_1=\sum_{i=1}^{\lceil k/2\rceil} a_i$
  and $v_2=\sum_{i=\lceil k/2\rceil+1}^k a_i$. By induction on $k$,
  there are $r_1=1+v_1\otimes \phi_1$ and $r_2=1+v_2\otimes\phi_2$ for
  some $\phi_1,\phi_2\in V^*$ such that $\ell_{Y_4}(r_1,r_2)\leq
  \lceil k/2\rceil^6$.  Using Lemma \ref{lem:Y4-one-way-neighbours}
  and Lemma \ref{lem:Y3-short-diameter} we get a path from $r_2$ to
  $r_1$ in $\Gamma(r_1,r_2,Y_4)$ satisfying the conditions in Lemma
  \ref{lem:bounded_graphs}. Thus, using Lemma \ref{lem:bounded_graphs}
  to this path, we get a transvection $r=1+v\otimes \phi$ with
  $\ell_{r_1,r_2,Y_4}(r)\leq c=28$. Thus, $\ell_{Y_4}(r)\leq
  \ell_{r_1,r_2,Y_4}(r)\cdot \ell_{Y_4}(r_1,r_2)\leq 28\cdot \lceil
  k/2\rceil^6\leq k^6$, as claimed. Let $Y_5$ be the union of all
  transvection groups in $Y_4^{(n^6)}$. Then the conclusion of the
  lemma holds for $Y_5$.
\end{proof}
\begin{lem}\label{lem:bounded_graphs_2}
  Let $\mathcal T$ be the set of all transvections. Then 
  $\ell_{Y_5}(\mathcal T)=O(1)$.
\end{lem}
\begin{proof}
  Let $0\neq v\in V,\ 0\neq \phi\in V^*$ satisfying $\phi(v)=0$.  We
  need to prove that $1+v\otimes \phi\in Y_5^c$ for some constant $c$.
  By Lemma \ref{lem:Y5-everything-is-image}, there are transvections
  $s_1,s_2\in Y_5$ such that $s_1=1+v_1\otimes \phi,\ s_2=1+v\otimes
  \phi_2$ for some $v_1\in V,\ \phi_2\in V^*$. If $\langle
  v_1\rangle=\langle v\rangle$ or $\langle \phi_2\rangle=\langle
  \phi\rangle$, then the assertion follows since $Y_5$ is
  $K$-closed. Furthermore, if $(s_1,s_2)$ is an edge, 
  then $[s_2^{1/\phi_2(v_1)},s_1]=1+v\otimes \phi$.
  So for the remainder we assume that $\langle
  v_1\rangle\neq\langle v\rangle,\ \langle \phi_2\rangle\neq\langle
  \phi\rangle$ and there is no edge
  between $s_1$ and $s_2$ in either direction.

  We claim that there is a path $s_1=r_1,r_2,\ldots, r_k=s_2$ with
  $k\leq 5$ such that $\ell_{Y_5}(r_1,\ldots,r_k)\leq 5$, and
  none of $(r_2,r_1),\ (r_{k},r_{k-1}),\ (r_k,r_2),\ (r_{k-1},r_1)$ are
  edges. First, let $r_1=s_1,\ r_k=s_2$ (The value of $k$ will be
  specified later).  Since $\langle v_1\rangle \neq \langle v
  \rangle$, there is a $\psi\in V^*$ such that $\psi(v_1)\neq 0,\
  \psi(v)=0$.  Now, by Lemma \ref{lem:Y5-everything-is-image},
  $1+u\otimes \psi\in Y_5$ for some $u\in V$. Furthermore, by Lemma
  \ref{lem:Y4-one-way-neighbours}, there is an $1+u'\otimes \psi'\in
  Y_4$ such that $(1+u\otimes\psi,1+u'\otimes \psi')$ is a one-way
  edge.  Then $[(1+(\psi'(u))^{-1}u'\otimes \psi',1+u\otimes
  \psi]=1+u'\otimes\psi\in Y_5^{4}$.  Let $U=\langle u, u'\rangle$, then 
  $\dim(U)=2$ and 
  \[
  \{1+x\otimes \psi\,|\,x\in U\}=(1+u\otimes
  \psi)^K\cdot (1+u'\otimes \psi)^K\subset Y_5^5.
  \] 
  The resctriction $\phi$ to $U$ must have non-trivial kernel, so
  there is an $0\neq x\in U$ with $\phi(x)=0$. Now, let
  $t=1+x\otimes \psi$. Then $\ell_{Y_5}(t)\leq 5$, $(r_1,t)$ is
  a one-way edge and $(r_k,t)$ is not an edge. In an analogous way, 
  we can prove the existence of a transvection $t'$ such that
  $\ell_{Y_5}(t')\leq 5$, $(t',r_k)$ is 
  a one-way edge and $(t',r_1)$ is not an edge. 
  Now, one of the following holds: 
  \begin{enumerate}
    \item[(a)] If $t^K=(t')^K$, then let $k=3$ and let $r_2$ be any 
      transvection satisfying $r_2^K=t^K=(t')^K$. 
    \item[(b)] If $(t,t')$ is an edge, then let $k=4$ and $r_2=t,\ r_3=t'$. 
    \item[(c)] The previous two cases does not hold. Then let $k=5$,
      $r_2=t,\ r_4=t'$ and $r_3\in Y_5$ a transvection such that
      $r_2,r_3,r_4$ is a path.  (The existence of such an $r_3$
      follows from Lemma \ref{lem:Y3-short-diameter}.)
  \end{enumerate}
  For each $1\leq i\leq k$, let $r_i=1+v_i\otimes \phi_i$ (where
  $\phi_1=\phi$ and $v_k=v$).  Now,
  \[
  [r_3^{1/\phi_3(v_2)},[r_2^{1/\phi_2(v_1)},r_1]]=1+v_3\otimes\phi_1=1+v\otimes
  \phi\textrm{ for }k=3
  \]
  and
  \[
  [r_4^{1/\phi_4(v_3)},[r_3^{1/\phi_3(v_2)},[r_2^{1/\phi_2(v_1)},r_1]]]
  =1+v_4\otimes\phi_1=1+v\otimes\phi\textrm{ for }k=4.
  \]
  Finally, let $k=5$. If $(r_1,r_4)$ or $(r_2,r_5)$ is a (one-way)
  edge, then the case $k=3$ can be used. Otherwise, if $(r_3,r_1)$ or
  $(r_5,r_3)$ is not an edge, then we can use $[r_2,r_1]$ or
  $[r_5,r_4]$ to reduce the problem of generating $1+v_5\otimes
  \phi_1$ to the case $k=4$. Then an iterated commutator similar as
  above works.  So for the remainder we assume that both $(r_5,r_3)$
  and $(r_3,r_1)$ are edges. Since $(r_1,r_2),\ (r_4,r_5)$ are one-way
  edges and none of $(r_1,r_4),\ (r_4,r_1),\ (r_2,r_5),\ (r_5,r_2)$ is
  an edge, there are $\lambda,\mu\in K$ such that $t_1=r_4^\mu
  r_2^\lambda r_3r_2^{-\lambda}r_4^{-\mu}$ satisfies that
  $r_5,t_1,r_1$ is a one-way directed path. 
  Using again Lemma \ref{lem:Y3-short-diameter}, there is a $t_2\in Y_5$
  such that $r_1,t_2,r_5$ is a path. Changing $t_2$ to $r_1t_2r_1^{-1}$ if 
  necessary, we can assume that $(t_1,t_2)$ is an edge. 
  Now, there is a $\lambda\in K$ such that $t_3=t_1^\lambda t_2t_1^{-\lambda}$ 
  satisfies that $r_1,t_3,r_5$ is a path and $(t_3,r_5)$ is a one-way edge. 
  Then $[[r_5^K,t_3],r_1]=(1+v_5\otimes \phi_1)^K=(1+v\otimes \phi)^K$ holds. 
\end{proof}
\begin{proof}[Proof of theorem \ref{thm:main_K}]
  Using Lemmas \ref{lem:strongly_connected},
  \ref{lem:Y3-short-diameter}, \ref{lem:Y4-one-way-neighbours},
  \ref{lem:Y5-everything-is-image}, \ref{lem:bounded_graphs_2},
  \ref{lem:standard_tv} and Corollary \ref{cor:Y2-one-way-edge} we get
  that
  \begin{multline*}
  \ell_X(SL(V))\leq \ell_X(Y_1)\ell_{Y_1}(Y_2)\ldots \ell_{Y_5}(\mathcal T)
  \ell_{\mathcal T}(SL(V))\\
  \leq O(n^2)\cdot O(n)\cdot O(n)\cdot O(1)\cdot O(n^6)\cdot O(1)
  \cdot O(n^2)=O(n^{12}).
  \end{multline*}
\end{proof}
\section{A generalisation of a Theorem of Humphries}\label{sec:Humphries}
As a side-effect of the proof of Theorem \ref{thm:main_K}, we
generalise and extend a theorem of Humphries.
In \cite{Humphries}, the author gave a sufficient and 
necessary condition when a set of transvections 
$S\subset SL(n,p)$ of size $n$ generates $SL(n,p)$. 
(Note that $n$ is the minimal possible size of a generating 
set $S$ of transvections in $SL(n,p)$.)

For this section, let $V$ be an $n>2$-dimensional vector space
over an arbitrary field $K$.
Let $S,S'\subset SL(V)$ be two sets of transvections (of the same size).  
According to Humphrey, we say that
there is a \emph{$t$-equivalence} $S\to S'$ if there is a chain
$S=S_0,S_1,\ldots,S_t=S'$ of sets of transvections such that for each
$1\leq i\leq t$ there are $x,y\in S_{i-1}$ such that $S_i$ is obtained
from $S_{i-1}$ by replacing $x$ to its conjugate $yxy^{-1}$.  Clearly,
if there is a $t$-equivalence $S\to S'$, then $S$ and $S'$ generate
the same subgroup of $SL(V)$.

Let $S=\{t_\alpha=1+v_\alpha\otimes \phi_\alpha\,|\,\alpha\in
I\}\subset SL(n,K)$ be any set of transvections.
We consider the following properties:
\begin{enumerate}
\item[(P 1.)]$\langle v_\alpha\,|\,\alpha\in I\rangle=V$ and 
  $\langle\phi_\alpha\,|\,\alpha\in I\rangle=V^*$.
\item[(P 2.)]$\Gamma(S)$ is strongly connected.
\item[(P 3.)] There is a non-singular cycle in $\Gamma(S)$.
\item[(P 3'.)] There is a $t$-equivalence $S\to S'$ such that 
  $\Gamma(S')$ contains a one-way directed edge. 
\end{enumerate}
Humphries result \cite[Theorem 1.1]{Humphries} says that 
a set of transvections $S\subset SL(n,p)$ of size $n$ generates 
$SL(n,p)$ if and only if $S$ satisfies (P 1.), (P 2.) and (P 3'.). 
One expects that the same assertion should be true without the condition 
$|S|=n$, but Humphries' proof uses this condition in an essential way. 

For a set of transvections $S=\{t_\alpha\,|\,\alpha\in I\}\subset SL(V)$, we
use the notation $S^K$ for the $K$-closure of $S$, i.e.
$S^K=\cup_{\alpha\in I}^m t_\alpha^K$.  Using the arguments of Section
\ref{sec:main_K}, we prove the following.
\begin{thm}\label{thm:Humphries-SL-gen}
  Let $S\subset SL(V)$ be a set of transvections. Then $S^K$
  generates $SL(V)$ if and only if $\{$(P\ 1.),(P\ 2.),(P\ 3.)$\}$ 
  or $\{$(P\ 1.),(P\ 2.),(P\ 3'.)$\}$ holds for $S$. 
\end{thm}
Note that if $|K|$ is a prime, then $t^K=\langle t\rangle$ for any
transvection $t$, so $S^K$ can be replaced with $S$ in this Theorem. 
On th other hand, if $L$ is a proper subfield of $L$, and 
$S\subset SL(n,L)\leq SL(n,K)$ then clearly $\langle S\rangle \neq SL(n,K)$ 
regardless of what conditions $S$ satisfy. 

First we prove the following
\begin{thm}\label{thm:Humphries-irred}
  Let $S\subset SL(V)$ be a set of transvections. Then $S$
  generates an irreducible subgroup of 
  $SL(V)$ if and only if $\{$(P\ 1.),(P\ 2.)$\}$ holds for $S$.
\end{thm}
\begin{proof}
  Let $S=\{t_\alpha=1+v_\alpha\otimes \phi_\alpha\,|\,\alpha\in I\}$ and let 
  $H=\langle S^K\rangle\leq SL(V)$. First, let us assume that 
  $H\leq SL(V)$ is irreducible. Then the $H$-invariant subspaces 
  $\langle v_\alpha\,|\,\alpha\in I\rangle>0$ and 
  $\cap_{\alpha\in I}\ker(\phi_\alpha)<V^*$ must be trivial, which proves (P 1.).
  Furthermore, (P 2.) also holds by the last paragraph of the proof
  of Lemma \ref{lem:strongly_connected}.

  Now, let us assume that $\{$(P\ 1.),(P\ 2.)$\}$ holds for $S$.  Let
  us assume that $0\neq U\leq V$ is an $H$-invariant subspace and let
  $0\neq u\in U$. By (P 1.), there is an $\alpha\in I$ such that
  $u\notin \ker(\phi_\alpha)$. Then $v_\alpha=t_\alpha(u)-u\in U$.
  Let $\beta\in I$ be any element of the index set.  Property (P 2.)
  implies the existence of a path $t_\alpha=t_{\gamma_0},
  t_{\gamma_1}, \ldots, t_{\gamma_k}=t_\beta$ in $\Gamma(S)$.  Using
  induction on $k$, it follows that $v_\beta=v_{\gamma_k}\in \langle
  t_{\gamma_k}(v_{\gamma_{k-1}})-v_{\gamma_{k-1}}\rangle \leq U$.
  Therefore, $U\geq \langle v_\alpha\,|\,\alpha\in I\rangle =V$ by
  property (P 1.).
\end{proof}
\begin{proof}{Proof of Theorem \ref{thm:Humphries-SL-gen}}
  Let $S\leq SL(V)$ be a set of transvections. 
  First let us assume that $S^K$ generates $SL(V)$. Then 
  $\langle S\rangle$ acts irreducibly on $V$, so 
  (P 1.) and (P 2.) follows by Theorem \ref{thm:Humphries-irred}. 
  It is easy to check that
  the proof of Lemma \ref{lem:Y1-non-singular} can be applied to $S^K$. Thus, 
  $\Gamma(S^K)$ contains a chordless non-singular cycle, which implies property 
  (P 3.). Finally, (P 3'.) is a consequence of 
  (P 3.) by Lemma \ref{lem:non-sing-to-one-way}.

  For the converse direction, let use assume that $S$ has properties
  $\{$(P 1.),(P 2.),(P 3.)$\}$.  Then $S$ also has property $\{$(P
  1.),(P 2.),(P 3'.)$\}$ by Lemma \ref{lem:non-sing-to-one-way}.
  Finally, one can check that after Corollary \ref{cor:Y2-one-way-edge}, our
  argument only uses that $Y_2$ has properties $\{$(P 1.),(P 2.)$\}$ and 
  $\Gamma(Y_2)$ contains a one-way directed edge, but it never refers to 
  the identity $\langle Y_2\rangle =SL(V)$. So this argument can be applied to 
  $S^K$ instead of $Y_2$ to prove that $\langle S^K\rangle =SL(V)$. 
\end{proof}
As a consequence of Theorems \ref{thm:Humphries-irred}, 
\ref{thm:Humphries-SL-gen} and the main result of \cite{McLaughlin} 
we get 
\begin{cor}
  Let us assume that $K\neq \FF 2$ and $S\subset SL(V)$ is a
  set of transvections. Then $S^K$ generates $Sp(V)$ is and only if 
  $S$ has properties (P 1.) and (P 2.) but it does not have (P 3.).
\end{cor}

\end{document}